\documentclass[11pt]{article}
\usepackage{bbm}
\usepackage{geometry}
\usepackage{amsfonts}
\usepackage{latexsym, amssymb, amsmath, amscd, amsthm, amsxtra}
\usepackage{mathtools}
\usepackage{enumerate}
\usepackage[all]{xy}
\usepackage{mathrsfs}
\usepackage{fancyhdr}
\usepackage{listings}
\usepackage{hyperref}

\hypersetup{
     colorlinks=true
}
\def\P{\mathbb{P}}

\def\E{\mathbb{E}}

\def\sfb{B}

\def\R{\mathbb{R}}

\def\11{\mathbf{1}}

\def\cq{1/10}

\newcommand\dif{\mathop{}\!\mathrm{d}}

\DeclarePairedDelimiterX{\inp}[2]{\langle}{\rangle}{#1, #2}

\newtheorem{thm}{Theorem}[section]

\newtheorem{proposition}[thm]{Proposition}

\newtheorem{lemma}[thm]{Lemma}
\newtheorem{cor}[thm]{Corollary}
\newtheorem{defn}[thm]{Definition}

\newtheorem{thmx}{Theorem}

\newtheoremstyle{cited}%
  {3pt}
  {3pt}
  {\itshape}
  {}
  {\bfseries}
  {.}
  {.5em}
  {\thmname{#1} \thmnumber{#2} \thmnote{\normalfont#3}}

\theoremstyle{cited}
\newtheorem{citedthm}[thmx]{Theorem}

\theoremstyle{definition}

\numberwithin{equation}{section}

\usepackage{xcolor}
\definecolor{qing}{RGB}{0, 153, 153}

\usepackage{environ}
\begin{document}
\title{Sharp threshold for the Ising perceptron model}
\author{Changji Xu\\[2ex]
Department of Statistics, The University of Chicago
}
\date{May 15, 2019}

\maketitle
\begin{abstract}
Consider the discrete cube $\{-1,1\}^N$ and a random collection of half spaces which includes each half space $H(x) := \{y \in \{-1,1\}^N: x \cdot y \geq \kappa \sqrt{N}\}$ for $x \in \{-1,1\}^N$ independently with probability $p$. Is the intersection of these half spaces empty? This is called the Ising perceptron model under Bernoulli disorder. We prove that this event has a sharp threshold; that is, the probability that the intersection is empty increases quickly from $\epsilon$ to $1- \epsilon$ when $p$ increases only by a factor of $1 + o(1)$ as $N \to \infty$.
\end{abstract}

\vspace{12pt}
\section{Introduction}
Consider the discrete cube $\Sigma_N = \{-1,1\}^N$. Let $\kappa \in \R$ be a constant, there is a half cube centered at each point $x \in \Sigma_N$,
\begin{equation}
		H(x) := \{y \in \Sigma_N: x \cdot y \geq \kappa \sqrt{N}\}\,.
\end{equation}
We now put an independent Bernoulli random variable $\omega_x$ with $\P(\omega_x = 1) = p$ at each point $x \in \Sigma_N$. This defines a probability measure $\P = \P_p$ on $\{0,1\}^{\Sigma_N}$.
We are interested in whether the intersection of half cubes centered at points in the random set $\{x : \omega_x = 1\}$ is empty, namely,  
\begin{equation}
\label{eq:intersection empty}
	 \bigcap_{x : \omega_x = 1} H(x) = \varnothing.
\end{equation}
This model is known as {\it the Ising perceptron under Bernoulli disorder}. It was conjectured that there is a critical value $p_{c} = p_c(\kappa,N)$, such that 
as $N \to \infty$, the probability that \eqref{eq:intersection empty} occurs tends to one when $p > (1+ \epsilon)p_{c} $, and tends to zero when $p < (1- \epsilon)p_{c} $. Moreover, the critical value $p_{c}$ is conjectured to be $\alpha_{c}(\kappa)N 2^{-N}$, where $\alpha_{c}(\kappa)$ is an explicit constant depending on $\kappa$. 

The main result of this paper resolves the first part of this conjecture. Let $p_N(\theta)$ be the probability such that 
\begin{equation}
	\P_{p_N(\theta)}\Big(\bigcap_{x : \omega_x = 1} H(x) = \varnothing \Big) = \theta\,.
\end{equation}
\begin{thm}
\label{thm:perceptron}
For any $\kappa \in \R$, $N$ sufficiently large, and all $\epsilon \geq (\log \log N)^{-\cq}$, we have
\begin{equation}
\label{eq:thm}
	 \frac{p_N(1 - \epsilon)}{p_N(\epsilon)} \leq 1 + (\log \log N)^{-\cq}\,.
\end{equation}
\end{thm}

\subsection{Ising perceptron}
The analysis of the Ising perceptron, along with various perceptron models, can be traced back to physicists' work thirty years ago \cite{G87,GD88,KM89,M89}, motivated by understanding the storage capacity of the neuron network.

The value of the threshold $p_{c} = \alpha_{c}(\kappa) N 2^{-N}$ was predicted by Krauth and Mezard in \cite{KM89}. Efforts have been made to provide rigorous bounds on $p_{c}$. It is known (see \cite{KR98} by Kim and Roche, \cite{Tala99} by Talagrand) that for $\kappa = 0$, there exist absolute constants $c,c_1, c_2 \in (0,1)$ such that for large $N$,
		\begin{equation}
			\label{eq:ublb}
			c_1 N2^{-N} \leq p_N(e^{-cN}) \leq p_N(1 - e^{-cN}) \leq c_2  N2^{-N}\,.
		\end{equation}
Recently, Ding and Sun \cite{DS18} proved that for $\kappa = 0$ and any $p \leq (1 - \epsilon)p_c$, the probability that \eqref{eq:intersection empty} does not occur stays away from $0$ as $N \to \infty$ for a variant of this model, where the half spaces are centered at points chosen uniformly from the $N$-dimensional sphere of radius $N^{1/2}$. Such a variant is called {\it the Ising perceptron under Gaussian disorder}, where the conjectured threshold is the same as that of our model. 

We also mention another variant, called {\it spherical perceptron}, where one asks how many random half spaces in $\R^N$ is needed to cover the whole space $\R^N$. This model has a different conjectured threshold. And the conjectured threshold value has been proved to be correct rigorously by Shcherbina and Tirozzi \cite{ST03}. See also Stojnic \cite{Sto13} for some results on these models.

\subsection{Sharp threshold theorems}
\label{sec:ST-thms}
The threshold phenomenon in the context of general monotone boolean functions has been studied extensively. Consider a probability space $(\{0,1\}^n,\mu_p)$ with
\begin{equation}
 	\mu_p(\{x\})  = p^{\sum_{i = 1}^n x_i} (1 - p)^{n - \sum_{i = 1}^n x_i}\,.
 \end{equation}
 Let $f:\{0,1\}^n \rightarrow \{0,1\}$ be an increasing boolean function, namely,
$$ x_i \geq y_i ~~\text{ for all }i = 1,...n ~~~\implies ~~~f(x) \geq f(y)\,.$$
It is discovered independently by Margulis \cite{M74} and Russo \cite{Russo81} that the derivative $\mathrm{d}\E_p[f]/\mathrm{d}p$ is proportional to the total influence $I_f(p)$ of the function $f$, defined as follows:
\begin{equation}
	I_f(p) := \sum_{i \in [n]}\P(f(X) \not = f(X_{\{i\}}))\,,
\end{equation}
where $[n] := \{1,2,...,n\}$, $X \sim \mu_p$, and $X_{\{i\}}$ is the same as $X$ except that its $i$-th component $X_i$ is replaced by an independent copy of $X_i$.

\begin{thmx}[Margulis-Russo formula]Let $f$ be an increasing boolean function,
\label{Margulis-Russo formula}
	\begin{equation}
		\frac{\mathrm{d}\E_p[f]}{\mathrm{d}p} = \frac{I_f(p)}{2p(1-p)}\,.
	\end{equation}
\end{thmx}
This motivates researchers to understand boolean functions with small total influence. And it is realized that, roughly speaking, there is a sharp threshold unless the boolean function depends much on very few coordinates. See Russo \cite{R82}, Kalai, Kahn, and Linial \cite{KKL88}, Bourgain, Kahn, Kalai, Katznelson, and Linial \cite{BKKKL92}, Talagrand \cite{T94}, Friedgut \cite{F98, F99}, Bourgain \cite{B99}, and Hatami \cite{H12}. 

Our model corresponds to the case when $p \asymp \log n /n$ with $n = |\Sigma_N|$. In fact, it is more complicated to characterize boolean functions with small total influence when $p$ is very close to $0$ or $1$. Friedgut \cite{F99} gave the following characterization, which holds for $p$ in its entire range and monotone boolean functions corresponding to graph properties. He conjectured that it is also true for general monotone boolean functions.
\begin{citedthm}[\cite{F99}]
	Let $f$ be a monotone boolean function corresponding to some graph property. If $p \leq 1/2$, then for any $\epsilon>0$ there exists a  monotone function $g$, such that $\mu_p(f \not = g) \leq \epsilon$ and all minimal elements in $\{x : g(x) = 1\}$ have at most $K$ many $1$'s, where $K$ is a positive number depending only on $(I_f(p),\epsilon)$.
\end{citedthm}
In the appendix to Friedgut's paper, Bourgain \cite{B99} gave an alternative characterization that works for general monotone functions. 
\begin{defn}
	\label{def-general}
 A subset of $S \subset [n]$ is said to be a $(1- \delta)$-boosting set if
	\begin{equation}
		\E_p [f(x)| x_S=(1,\ldots,1)] \ge 1 -\delta\,,
	\end{equation}
	where $x_S := (x_i)_{ i\in S}$.
\end{defn}	

\begin{citedthm}[\cite{B99}]
Let $f$ be an increasing boolean function. There exist $\epsilon,K>0$ depending only on $(\E_p[f],I_f(p))$ such that if $p< \epsilon$,  then there exists a  $(\E_p[f] + \epsilon)$-boosting set of cardinality at most $K$.
\end{citedthm}
In fact (see e.g. \cite{D14}), Bourgain's proof works for low-influence functions on any product space. Also, one can choose $(\epsilon,K)$ so that the probability that $\{i \in [n]: x_i = 1\}$ contains a $(\E_p[f] + \epsilon)$-boosting set is at least $\epsilon$.

Later, Hatami \cite{H12} significantly generalized Friedgut and Bourgain's work. 
To state Hatami's theorem, we need to define the ``pseudo-junta''. Let $(\Omega^n,\pi^n)$ be a product probability space and $\mathcal J = (J_S)_{S \subset [n]}$ be a collection of boolean functions defined on $(\Omega^n,\pi^n)$ where $J_S$ depends only on $x_S$. Define $J_{\mathcal J}(x) := \cup_{S \subset [n], J_S(x) = 1} S$ and let $\mathcal{F}_\mathcal{J}$ be the $\sigma$-field induced by $x \mapsto (J_\mathcal J(x), x_{J_\mathcal J(x)})$. 
\begin{defn}
\label{pseudo-junta}
For $k>0$, a $k$-pseudo-junta is a boolean function $h$ which is measurable with respect to $\mathcal{F}_\mathcal{J}$ for some $\mathcal J$ satisfying $\int |J_{\mathcal J}| \dif \pi^n \leq k$.  
\end{defn}

\begin{citedthm}[\cite{H12}]
\label{Hatami}
	Let $f$ be a boolean function defined on a product probability space $(\Omega^n,\pi^n)$. Then for any $\epsilon>0$, there exists a $K$-pseudo-junta $h$ such that $\pi^n(f \not = h) \leq \epsilon$, where $K>0$ depends only on $(\epsilon, I_f)$.
\end{citedthm}
When $f$ is an increasing function defined on $(\{0,1\}^n,\mu_p)$, Hatami proved (see \cite[Corollary 2.10]{H12}) the existence of $(1 - \delta)$-boosting set of cardinality at most $ K = K(\E_p[f],\delta, I_f)$ for any $\delta>0$. We point out that one can modify the proof of this result to get the following stronger result, which says that for any $\epsilon,\delta \in (0,1)$, conditioned on the event $\{f = 1\}$, the set $\{i \in [n]: x_i = 1\}$ contains a $(1- \delta)$-boosting set of cardinality at most $K$ with probability $1 - \epsilon$, where $K$ depends only on $(\epsilon \E_p[f],\delta,I_f(p))$. Such a property is needed in the proof of Theorem \ref{thm:perceptron}.

\begin{proposition}
\label{01Sharpthm}Let $f$ be an increasing boolean function. Denote by $\sfb_{\delta,k}$ the event that $\{i \in [n]: x_i = 1\}$ contains a $(1- \delta)$-boosting set of cardinality at most $k$.
If $p \leq 1/2$, then for any $\epsilon,\delta>0$, the following holds for $k = \exp(10^{12} \lceil I_f(p) \rceil^2 (\delta\epsilon)^{-2})$,
	\begin{equation}
			\mu_p\big(\{f(x) = 1\} \setminus \sfb_{\delta,k} \big) \leq  \epsilon\,.
	\end{equation}
\end{proposition}
The proof of this proposition is provided in Section \ref{sec:STT}.
\section{Proof Outline}
In this section, we provide a detailed outline of the proof for the main theorem, while proofs of a number of lemmas are postponed to later sections.

We consider $N$ large and $\epsilon_N \in [ (\log \log N)^{-\cq}, 1)$ for the rest of the paper. Let
\begin{equation}
\label{eq:def-p-star}
	p_* := \inf \left \{p \geq p_N(\epsilon_N) : I(p) \leq 8  (\log\log N)^{\cq} \right\}\,.
\end{equation}

\begin{lemma}
\label{pstar - ub}
$ p_N(\epsilon_N) \leq p_* \leq p_N(\epsilon_N) (1 + (\log \log N)^{-\cq}/2)$.
\end{lemma}
\begin{proof}
By Theorem \ref{Margulis-Russo formula} (Margulis-Russo formula), 
\begin{equation*}
\begin{split}
		\P_{p_*}\Big(\bigcap_{x : \omega_x = 1} H(x) = \varnothing \Big) &= \epsilon_N + \int_{p_N(\epsilon_N)}^{p_*} \frac{I(p)}{2p(1-p)} \dif p \geq \big(p_* -  p_N(\epsilon_N) \big) \, \frac{8(\log\log N)^{\cq}}{2p_*}\,.
\end{split}
\end{equation*}
This yields the desired result.
\end{proof}
The following result is an analogue of \eqref{eq:ublb} for general $\kappa$. It will be proved in Section \ref{sec:ublb} by modifying Talagrand's proof in \cite{Tala99}.
\begin{lemma}
\label{ublb-kappa}
	There exists constants $c, c_3, C_1>0$ depending only on $\kappa$, such that for large $N$,
			\begin{equation} 
			\label{eq:ublb-kappa}
			c_3 N2^{-N} \leq p_N(e^{-cN})\leq p_N(1 - e^{-cN})  \leq C_1  N2^{-N}\,.
		\end{equation}
\end{lemma}
\subsection{Existence of boosting sets}

To utilize the sharp threshold theorem, we define the boosting sets of our model.

\begin{defn}
\label{def-boost-per}
A set $\{x_1,\dots,x_k\} \subset \Sigma_N$ is said to be a $(1 - \delta)$-boosting set if
\begin{equation}
	\P_{p_*}\Big(\bigcap_{x: \omega_x = 1} H(x)  \cap \bigcap_{j = 1}^{k} H(x_j) = \varnothing\Big) \geq 1 - \delta\,.
\end{equation}
\end{defn}
This is consistent with Definition \ref{def-general}; In our model, $n = |\Sigma_N|$ and $i \in [n]$ corresponds to $x \in \Sigma_N$. Also, $\{x \in \Sigma_N: \omega_x = 1\} \cup \{x_1,\dots,x_k\}$ has the same distribution as $\{x \in \Sigma_N: \omega_x = 1\}$ conditioned on $\{\omega_{x_i} = 1 \text{ for }i=1,\dots,k\}$. 

We now apply Proposition \ref{01Sharpthm} to our model to get the following result about $(1 - \delta)$-boosting sets.
\begin{lemma}
\label{SHT}
For any $\delta \in (0,1)$, with $\P_{p_*}$-probability at least $\epsilon_N/2$, there exists a subset of $\{x : \omega_x = 1\}$ which is a $(1- \delta)$-boosting set of cardinality
\begin{equation}
\label{eq:size k}
	k \leq \exp(10^{15}(\delta \epsilon_N)^{-2} (\log\log N)^{1/5} )\,. 
\end{equation}
\end{lemma}
\begin{proof}
Recall that $\sfb_{\delta,k_0}$ is the event that $\{x : \omega_x = 1\}$ contains a $(1- \delta)$-boosting set of cardinality at most $k_0$. By Proposition \ref{01Sharpthm} and the definition \eqref{eq:def-p-star}, we get
\begin{equation}
	\P_{p_*}\Big( \Big\{\bigcap_{x: \omega_x = 1} H(x) = \varnothing\Big\} \cap B_{\delta,k_0}^c \Big) \leq \epsilon_N/2
\end{equation}
for $k_0 = \exp(10^{12} \lceil I(p_*) \rceil^2 (\delta\epsilon_N/2)^{-2}) \leq \exp(10^{15}(\delta \epsilon_N)^{-2} (\log\log N)^{1/5} )$. Therefore,
\begin{equation*}
	\P_{p_*}\big( B_{\delta,k_0} \big) \geq\P_{p_*}\Big( \Big\{\bigcap_{x: \omega_x = 1} H(x) = \varnothing\Big\} \cap B_{\delta,k_0} \Big) \geq \P_{p_*}\Big(\bigcap_{x: \omega_x = 1} H(x) = \varnothing\Big) -\epsilon_N/2\,.
\end{equation*}
This is at least $\epsilon_N/2$ by $p_* \geq p_N(\epsilon_N)$ as in \eqref{eq:def-p-star}, and thus we complete the proof.
\end{proof}
We fix a small $\delta>0$. What this lemma tells us is that not only is there a $(1- \delta)$-boosting set, we actually have plenty of them. In addition, we will prove that, roughly speaking, most of the sequence of length $k$ in $\Sigma_N$ are the same, up to some automorphisms and minor changes on each vector. Therefore, we wish that $k$ independent uniform random vectors in $\Sigma_N$ would have a similar effect as a $(1- \delta)$-boosting set. And we will prove that by increasing $p_*$ to $p_*+\Delta$ for some small $\Delta$ (which will produce roughly $\Delta 2^N$ more independent half spaces), the probability that the intersection is empty will become very close to one, which will yield $p_* + \Delta \geq p_N(1 - \epsilon_N)$. To make this precise, we need some notations for the automorphisms on $\Sigma_N$.

 \subsection{Automorphisms on $\Sigma_N$}
 \label{sec:auto}
To be able to make use of the symmetry of our model, we introduce the notations of two types of automorphisms on $\Sigma_N$. 

The first type is sign switching, which will be denoted by $g = (g^1,\dots,g^N) \in \Sigma_N$. For all $x =(x^1,\dots,x^N)\in \Sigma_N$, we define
\begin{equation}
	g \circ x := (g^1 x^1,\dots,g^N x^N)\,.
\end{equation}

The second type is label exchanging, which will be denoted by permutations in the symmetric group $S_N$. For all $\sigma \in S_N$ and $x \in \Sigma_N$, we define
\begin{equation}
	\sigma \circ x := (x^{\sigma^{-1}(1)},\dots, x^{\sigma^{-1}(N)})\,.
\end{equation}

When an automorphism acts on a subset of $\Sigma_N$ or a sequence of vectors in $\Sigma_N$, it acts on each member in that set or sequence. Namely, for any $g \in \Sigma_N \cup S_N$,  we define
\begin{equation}
\begin{split}
	g \circ (x_1,\dots,x_k)  &:= (g \circ x_1,\dots, g \circ x_k)\,, \quad \text{for all}~(x_1,\dots,x_k) \in \Sigma_N^k, k \geq 1\,,\\
	g \circ A  &:= \{g \circ x : x \in A\}\,, \quad \text{for all}~A \subset \Sigma_N\,.	
\end{split}
\end{equation}
The following lemma says under any aforementioned automorphism $g$, the half cube centered at $x$ will be mapped to the half cube centered at $g \circ x$ for all $x \in \Sigma_N$.
\begin{lemma}
\label{autoH}
	For any $g,x \in \Sigma_N$, permutation $\sigma \in S_N$,
	\begin{equation}
		H(g\circ x) = g\circ H(x),\quad H(\sigma \circ  x) = \sigma \circ H(x)\,.
	\end{equation}
\end{lemma}
Fix $k \geq 1$. Suppose that we are given a sequence of vectors $\mathbf Y = (y_1,\dots,y_k) \in \Sigma_N^k$. We want to record $\mathbf Y$ in the following way. We first record $y_1$. And then consider $y_1 \circ y_2,...,y_1 \circ y_k$, which are the coordinates of $y_2,...,y_k$ relative to $y_1$. It can be viewed alternatively as a sequence of length $N$ with each component taking value in $\{-1,1\}^{k-1}$. Hence, it suffices to record for each $a = (a_1,...,a_{k-1}) \in \{-1,1\}^{k-1}$, the positions of $a$ in this sequence. More precisely, we let
\begin{equation}
\begin{split}
	I_a(\mathbf Y) &:= \{ j \in [N]: \text{the $j$-th component of $y_1 \circ y_i$ is $ a_{i-1}$ for } i = 2,\dots,k\}\,,\\
	\text{and}\quad \varphi_I(\mathbf Y) &:= (I_a(\mathbf Y))_{a \in \{-1,1\}^{k-1}}\,.
\end{split}
\end{equation}
\begin{lemma}
\label{bijection}
The mapping $\mathbf Y \mapsto (y_1, \varphi_I(\mathbf Y))$ is a bijection.
\end{lemma}
We introduce $\varphi_I(\mathbf Y)$ for the following two reasons. On one hand, $\varphi_I(\mathbf Y)$ is invariant under sign switching (see Lemma \ref{phiI-inv} below). On the other hand, two sequences are equivalent up to some automorphisms if the cardinalities of $(I_a)_{a \in \{-1,1\}^{k-1}}$ for two sequences are the same (See Lemma \ref{number match} below).
\begin{lemma}
\label{phiI-inv}
	For any $g \in \Sigma_N$, $\varphi_I(g \circ \mathbf X) = \varphi_I(\mathbf X)$.
\end{lemma}
\begin{lemma}
\label{number match}
	If  $|I_a(\mathbf X)| = |I_a(\mathbf Y)|$ for all $a \in \{-1,1\}^{k-1}$, then there exists a permutation $\sigma \in S_N$ and $g \in \Sigma_N^k$ such that $
	g \circ (\sigma \circ \mathbf X) = \mathbf Y$.
\end{lemma}
The proofs of the lemmas in this subsection are postponed to Section \ref{sec:proof-auto}.

\subsection{Converting random vectors to a boosting set}
\label{sec:map}
\subsubsection{Admissible sequences}
By the law of large numbers, we expect that most $\mathbf X \in \Sigma_N^k$ satisfy $|I_a(\mathbf X)| \approx 2^{-(k-1)}N$ for all $a \in \{-1,1\}^{k-1}$. We will call such $\mathbf X$'s admissible. Since we have plenty of $(1 - \delta)$-boosting sets, some of them must be admissible.

\begin{defn}
Let $C_2>0$ be a constant to be determined in Lemma \ref{typical prob}. A sequence of vectors $\mathbf X = (x_1,x_2,\dots,x_k) \in \Sigma_N^k$ is said to be admissible if for any $a \in \{-1,1\}^{k-1}$,
\begin{equation}
	\big||I_a(\mathbf X)| - 2^{-(k-1)}N\big| \leq C_2 (N\log N)^{1/2}.
\end{equation}
\end{defn}

\begin{lemma}
\label{typical prob}Let $1 \leq k \leq \log N/2$ and $Y_1, \dots Y_k$ be i.i.d. uniform random vectors in $\Sigma_N$ defined on a probability space $(\Omega,\mu)$. Then there exists an absolute constant $C_2>0$ such that for $N$ sufficiently large depending on $C_2$,
	\begin{equation}
	\label{eq:1007}
		\mu(Y_1, \dots Y_k \text{ is not admissible}) \leq N^{-2^k}\,.
	\end{equation}	
\end{lemma}
\begin{proof}
First note that for any fixed $a \in \{-1,1\}^{k-1}$ the events
\begin{equation*}
	\{(Y^j_1Y^j_2,\dots,Y^j_1Y^j_k) = a\}
\end{equation*}
for $j = 1,...,N$ are independent and each event has probability $2^{-(k-1)}$. Then it follows from Chernoff's inequality (see e.g. \cite[Exercise 2.3.5]{vershynin2018high}) and the condition $k \leq \log N/2$ (which implies $2^{-(k-1)}N \geq N^{1-\log 2/2} \geq N^{0.51}$) that for any $C_2>0$ and $N$ sufficiently large depending on $C_2$,
\begin{equation}
	\mu\Big (| \sum_{j = 1}^{N}\11_{\{(Y^j_1Y^j_2,\dots,Y^j_1Y^j_k) = a\}} -2^{-(k-1)}N| \geq C_2(N\log N)^{1/2}\Big) \leq 2 N^{-c2^{k-1} C_2}\,,
\end{equation}
where $c$ is an absolute constant. Taking a union bound over all $a \in \{-1,1\}^{k-1}$ with a sufficiently large $C_2$ yields \eqref{eq:1007}.
\end{proof}
\begin{cor}
\label{typical boosting}
For any $\delta \in (0,1)$, there exists an admissible $(1 - \delta)$-boosting set of cardinality $k$ that satisfies \eqref{eq:size k}.
\end{cor}
\begin{proof}
For any $m \geq 1$, we have that
\begin{equation}
\label{eq:241352}
\begin{split}
	&\E_{p_*} \Big[ \Big| \Big\{(x_1,\dots,x_m) \in \Sigma_N^m : \substack{\omega_{x_i} =1, x_i \not = x_j \text{ for }i \not =j,\\ (x_1,\dots,x_m) \text{ not admissible}} \Big\}\Big|\Big]\\
	 &\leq p_*^m \times \Big| \Big\{(x_1,\dots,x_m) \in \Sigma_N^m :  (x_1,\dots,x_m) \text{ not admissible}\Big\}\Big|\leq (p_* 2^N)^m N^{-2^m}\,,
\end{split}
\end{equation}
where we used Lemma \ref{typical prob} in the last step. Note that combining Lemma \ref{pstar - ub} and Lemma \ref{ublb-kappa} implies 
\begin{equation}
\label{eq:useub}
	(p_* 2^N)^m N^{-2^m} \leq (2p_N(\epsilon_N)\cdot 2^N)^m N^{-2^m} \leq  (2C_1)^mN^{-2^m+m}\,.
\end{equation}
 Hence summing \eqref{eq:241352} over $m \geq 1$ yields that the probability that all subset of $\{x : \omega_x = 1\}$ are admissible is at least $1 -  N^{-1/2}$. Combined with Lemma \ref{SHT} and the condition $\epsilon_N \geq (\log\log N)^{-\cq}$, this yields the desired result. 
\end{proof}
\subsubsection{Convert gently}
Being admissible means that $|I_a|$'s are almost fixed. Recall that Lemma \ref{number match} says two sequences are equivalent up to some automorphisms if $(|I_a|)_{a \in \{-1,1\}^{k-1}}$ of two sequences are the same. Hence, we expect that we can convert one admissible sequence to another by first applying a gentle mapping, which perturbs each vector only by a little bit, and then applying some automorphisms.

We will also require the gentle mappings to satisfy certain symmetric property \eqref{eq:f-symmetric}, so that vaguely speaking it does not favor any particular direction.
\begin{defn}
A mapping $f=(f_1,f_2,\dots,f_k)$ where $f_i :\Sigma_N^k \to \Sigma_N$ for $1 \leq i \leq k$ is said to be gentle if for any $g \in \Sigma_N$ and $\mathbf Y = (y_1,\dots,y_k)\in \Sigma_N^k$,
\begin{equation}
\label{eq:f-symmetric}
	f(g \circ \mathbf Y) =g \circ f(\mathbf Y)\,,
\end{equation}
and for all $1 \leq i \leq k$,
\begin{equation}
\label{eq:f-close}
	\mathrm{dist}(y_i,f_i(\mathbf Y)) \leq 2^kC_2(N \log N)^{1/2}\,,
\end{equation}
where 
\begin{equation}
\label{eq:def-HMdist}
	\mathrm{dist}(u,v) := |\{i \in [N]: u_i \not = v_i\}|\,.
\end{equation}
\end{defn}
\begin{lemma}
\label{adjust}
Let $k \geq 1$ and $\mathbf X = (x_1,x_2,\dots,x_k) \in \Sigma_N^k$ be admissible. Then there exists a gentle mapping $f=(f_1,f_2,\dots,f_k)$ from $\Sigma_N^k$ to $\Sigma_N^k$ such that for any admissible $\mathbf Y \in \Sigma_N^k$,  there exists a permutation $\sigma \in S_N$ and $g \in \{-1,1\}^N$ such that $g \circ(\sigma \circ f(\mathbf Y)) = \mathbf X$.
\end{lemma}

Inspired by previous lemma and Lemma \ref{typical prob}, Corollary \ref{typical boosting}, we hope that $k$ uniformly random vectors in $\Sigma_N$ can effectively remove points in $\bigcap_{x : \omega_x = 1} H(x)$, because $k$ uniformly random vectors are very close to a $(1 - \delta)$-boosting set, which can remove all points in $\bigcap_{x : \omega_x = 1} H(x)$ with probability $1 - \delta$. The following lemma describes how much such removing ability can be inherited from a gentle mapping.
\begin{lemma}
\label{gentlemap}
	Let $1 \leq k \leq \log N/2$, $f$ be a gentle mapping, and $(Y_i)_{i \geq 1}$ be a sequence of i.i.d. uniform random vectors in $\Sigma_N$ defined on a probability space $(\Omega,\mu)$. Denote $\mathbf Y := (Y_1, \dots Y_k)$. For any $A \subset \{-1,1\}^N$, define
	\begin{equation}
\label{eq:def-qa}
	 	q(A) := \mu\Big( A \cap \bigcap_{i = 1}^k H(f_i(\mathbf Y))  = \varnothing\Big) \,.
	 \end{equation}
If $q(A) \geq (\log N)^{-1/3}$, then
	\begin{equation}
	\label{eq:64732}
		\mu\Big(A \cap \bigcap_{i = 1}^{kN_*} H(Y_i) = \varnothing\Big) \geq 1 - \exp(-c q(A) N_*)\,,
	\end{equation}
where $c>0$ is an absolute constant and 
\begin{equation}
	N_* := \lfloor N/ (\log N)^{1/2} \rfloor\,.
\end{equation}
\end{lemma}
We postpone the proofs of Lemmas \ref{adjust} and \ref{gentlemap} to Section \ref{sec:proof-map}, and provide the proof of Theorem \ref{thm:perceptron} in the next section assuming these lemmas.
\section{Proof of Theorem \ref{thm:perceptron}}
Set $\delta := \epsilon_N^2/36$. Let $\mathbf X = (x_1,x_2,\dots,x_k) \in \Sigma_N^k$ be a fixed admissible $(1 - \delta)$-boosting set as in Corollary \ref{typical boosting}, where
\begin{equation}
	 k \leq \exp(10^{15}(\delta \epsilon_N)^{-2} (\log\log N)^{1/5} ) \leq \exp((\log\log N)^{9/10} ).
\end{equation}
Let $f$ be the gentle mapping corresponding to $\mathbf X$ as in Lemma \ref{adjust}. Let $(Y_i)_{i \geq 1}$ be a sequence of i.i.d. uniform random vectors in $\Sigma_N$ defined on a probability space $(\Omega,\mu)$  and denote $\mathbf Y := (Y_1, \dots Y_k)$.  Define permutation $\sigma_{Y} \in S_N$ and $g_{Y} \in \{-1,1\}^N$ depending on $(\mathbf {X,Y}, f)$ such that $g_{Y} \circ(\sigma_{Y} \circ f(\mathbf Y)) = \mathbf X$ for all admissible $\mathbf Y$. 

We first prove that for typical $\bigcap_{x: \omega_x = 1} H(x)$, the probability that it has a non-empty intersection with the random set $ \bigcap_{i = 1}^k H(f_i(\mathbf Y))$ is negligible, more precisely,\begin{equation}
\label{eq:cond omega}
	\P_{p_*}\Big[q\Big(\bigcap_{x: \omega_x = 1} H(x)\Big) \leq \delta^{1/2}\Big] \leq 2\delta^{1/2}\,.
\end{equation}
To this end, we first note that by Lemma \ref{autoH}, for any $A \subset \Sigma_N$,
\begin{equation*}
	q(A) = \mu\Big( A \cap \bigcap_{i = 1}^k H(f_i(\mathbf Y)) = \varnothing\Big) \geq \mu\Big( \big(g_{Y}\circ (\sigma_{Y} \circ A) \big)\cap \bigcap_{i = 1}^k H(x_i) = \varnothing \Big) - \mu(\mathbf Y \text{ not admissible})\,.
\end{equation*}
Substitute $A = \bigcap_{x: \omega_x = 1} H(x)$ and take expectation. By Lemma \ref{typical prob}, we have
\begin{equation}
	\E_{p_*}\Big[q\Big(\bigcap_{x: \omega_x = 1} H(x)\Big)\Big] \geq \mu \otimes \P_{p_*}\Big( \bigcap_{x: \omega_x = 1} H(g_{Y}\circ (\sigma_{Y} \circ x))\cap \bigcap_{i = 1}^k H(x_i)  = \varnothing\Big) - N^{-2^k}\,.
\end{equation}
Note that $(\mathbf X,f)$ is fixed, $\mathbf Y$ and $\omega_x$'s are independent. We have that  $\{g_{Y}\circ (\sigma_{Y} \circ x):\omega_x = 1\}$ and $\{x:\omega_x = 1\}$ have the same distribution. Together with the facts that $\{x_1,\dots,x_k\}$ is a $(1 - \delta)$-boosting set and $\delta = \epsilon_N^2/36  \geq (\log\log N)^{-1/4}$, this yields
\begin{equation}
	\E_{p_*}\Big[q\Big(\bigcap_{x: \omega_x = 1} H(x)\Big)\Big]  \geq \mu \otimes\P_{p_*}\Big( \bigcap_{x: \omega_x = 1} H(x)\cap \bigcap_{i = 1}^k H(x_i)  = \varnothing\Big)  - N^{-2^k}\geq 1 - 2\delta\,.
\end{equation}
Hence \eqref{eq:cond omega} follows.

Next, recall that $N_* = \lfloor N/ (\log N)^{1/2} \rfloor$. Combining Lemma \ref{gentlemap} and \eqref{eq:cond omega} yields that
\begin{equation}
\label{eq:kNmore}
	\mu \otimes \P_{p_*}\Big(\bigcap_{x: \omega_x = 1} H(x) \cap \bigcap_{j = 1}^{kN_*} H( Y_j)  = \varnothing\Big)\geq 1-  2\delta^{1/2} - \exp(-c \delta^{1/2} N_*) \geq 1 - 3\delta^{1/2} = 1 - \epsilon_N/2\,.
\end{equation}
This tells us that the intersection will be empty with high probability after adding $kN_*$ more random half spaces. In fact, increasing $p_*$ should have a similar effect as adding more random half spaces. Indeed, set
\begin{equation}
	\Delta: =  (\log \log N)^{-\cq}p_N(\epsilon_N) /2\,.
\end{equation}
We claim that \eqref{eq:kNmore} implies 
\begin{equation}
\label{eq:0153}
\P_{p_* + \Delta}\Big(\bigcap_{x: \omega_x = 1} H(x) = \varnothing\Big)  \geq 1 - \epsilon_N\,,
\end{equation}
which together with Lemma \ref{pstar - ub} yields \eqref{eq:thm}.

To prove \eqref{eq:0153}, we let $\omega'_x$ for $x \in \Sigma_N$ be i.i.d independent Bernoulli random variables coupled with $(\omega_x)_{x\in \Sigma_N}$ such that $$\P_{p_*}(\omega'_x = 1) = p_* + \Delta, \quad \omega_x  \leq \omega'_x, \quad (\omega'_x,\omega_x)\text{, $x \in \Sigma_N$ are mutually  independent}\,.$$ Then given $\big(|\{x:\omega_x = 0,\omega'_x = 1\}|, \{x: \omega_x = 1\}\big)$, the conditional distribution of $\{x: \omega_x = 0,\omega'_x = 1\}$ is the uniform distribution on all subset of $\{x: \omega_x = 0\}$ of cardinality $|\{x:\omega_x = 0,\omega'_x = 1\}|$, which, ordered by subset inclusion, stochastically dominates $A \setminus \{x: \omega_x  = 1\}$, where $A$ is a random set that consists of $|\{x:\omega_x = 0,\omega'_x = 1\}|$ independent uniform random vectors in $\Sigma_N$.
Therefore
\begin{equation}
\label{eq:com123}
\begin{split}
	\P_{p_* + \Delta}\Big(\bigcap_{x: \omega_x = 1} H(x) = \varnothing\Big) \geq &~\mu \otimes \P_{p_*}\Big(\bigcap_{x: \omega_x = 1} H(x) \cap \bigcap_{j = 1}^{kN_*} H( Y_j)  = \varnothing\Big)\\
	& - \P_{p_*}\big(|\{x:\omega_x = 0,\omega'_x = 1\}| < k N_* \big) \,.
\end{split}	
\end{equation}
In addition, note that Lemma \ref{ublb-kappa} implies 
\begin{equation}
\label{eq:uselb}
	\Delta \cdot 2^N \geq c_3 (\log \log n)^{1/10}N/2> 2 k N_*\,.
\end{equation}
Hence Chernoff's inequality (see e.g. \cite[Exercise 2.3.5]{vershynin2018high}) yields
\begin{equation}
	\P_{p_*}(|\{x:\omega_x = 0,\omega'_x = 1\}| < k N_*) \leq e^{-c\Delta 2^N} \leq e^{-c(\log \log N)^{-\cq}N} \leq \epsilon_N/2\,,
\end{equation}
which together with \eqref{eq:com123} and \eqref{eq:kNmore} yields \eqref{eq:0153}, and then Theorem \ref{thm:perceptron} follows.\qed

\section{Properties of the automorphisms}
\label{sec:proof-auto}

\begin{proof}[\bf Proof of Lemma \ref{autoH}]For any $x,y \in \Sigma_N$ and $g \in \Sigma_N$, we have $g = g^{-1}$, hence
	\begin{equation}
		y \cdot (g \circ x) = \sum_{i = 1}^N y^ig^ix^i = (g\circ y)\cdot x = (g^{-1}\circ y)\cdot x\,.
	\end{equation}
Then 
\begin{equation}
\begin{split}
		H(g\circ x) &= \{y \in \Sigma_N : y \cdot (g \circ x) \geq \kappa \sqrt{N}\} = \{y \in \Sigma_N : (g^{-1}\circ y)\cdot x \geq \kappa \sqrt{N}\} \\
		&=  \{(g\circ y) \in \Sigma_N : y\cdot x \geq \kappa \sqrt{N}\} = g\circ H(x).\\
\end{split}
\end{equation}
For any $\sigma \in S_N$,
		\begin{equation}
		y \cdot (\sigma \circ x) = \sum_{i = 1}^N y^ix^{\sigma^{-1} (i)}= \sum_{i = 1}^N y^{\sigma (i)}x^{i} = (\sigma^{-1}\circ y)\cdot x \,.
	\end{equation}
Hence, as it is in the the previous case, $H(\sigma \circ  x) = \sigma \circ H(x)$.
\end{proof}

\begin{proof}[\bf Proof of Lemma \ref{bijection}]Note that $y_1 \circ (y_1\circ y_i) = y_i$,  the mapping $\mathbf Y \mapsto (y_1,(y_1\circ y_i)_{2 \leq i\leq k}) $ is a bijection. Also, it follows from the definition that $(y_1\circ y_i)_{2 \leq i\leq k}\mapsto \varphi_I(\mathbf Y)$ is a bijection.
\end{proof}

\begin{proof}[\bf Proof of Lemma \ref{phiI-inv}]
	For $2 \leq i \leq k$, we have $((g \circ y_1)\circ (g \circ y_i)) = (y_1\circ y_i)$. Hence $\varphi_I(g \circ \mathbf X) = \varphi_I(\mathbf X)$.
\end{proof}

\begin{proof}[\bf Proof of Lemma \ref{number match}]
	If  $|I_a(\mathbf X)| = |I_a(\mathbf Y)|$ for all $a \in \{-1,1\}^{k-1}$, then there exists $\sigma \in S_N$ such that for every $a \in \{-1,1\}^{k-1}$, $\sigma|_{I_a(\mathbf X)}$ is a bijection from $I_a(\mathbf X)$ to $I_a(\mathbf Y)$. Since for any $a \in \{-1,1\}^{k-1}$, 
	\begin{equation*}
	\begin{split}
		\sigma(I_a(\mathbf X)) &= \{ \sigma(j) \in [n]:x_1^j  x^j_i =  a_{i-1}, \forall ~2 \leq i \leq k\} \\
		&= \{ j \in [n]:x_1^{\sigma^{-1}(j)}  x^{\sigma^{-1}(j)}_i =  a_{i-1}, \forall ~2 \leq i \leq k\} = I_a(\sigma \circ \mathbf X)\,,
	\end{split}
	 \end{equation*} we have that
	\begin{equation}
		\varphi_I(\sigma \circ \mathbf X) =(I_a(\sigma \circ \mathbf X))_{a \in \{-1,1\}^{k-1}} = (\sigma(I_a( \mathbf X)))_{a \in \{-1,1\}^{k-1}} =  \varphi_I(\mathbf Y)\,.
	\end{equation}
Now set $g = y_1 \circ (\sigma \circ x_1)$. So $g\in \Sigma_N$. Then $g \circ (\sigma \circ x_1) = (y_1 \circ (\sigma \circ x_1)) \circ (\sigma \circ x_1) = y_1$ and by Lemma \ref{phiI-inv}, $\varphi_I(g \circ (\sigma \circ \mathbf X)) = \varphi_I(\sigma \circ \mathbf X) = \varphi_I(\mathbf Y)$. Therefore by Lemma \ref{bijection}, $g \circ (\sigma \circ \mathbf X) = \mathbf Y$.
\end{proof}

\section{Properties of gentle mapping}
\label{sec:proof-map}
\begin{proof}[\bf Proof of Lemma \ref{adjust}]
By Lemma \ref{bijection}, in order to define $f = (f_1, \ldots, f_k)$, it suffices to define $f_1(\mathbf Z)$ and $\varphi_I(f(\mathbf Z))$ for all $\mathbf Z\in \Sigma_N^k$. In light of this, we let
\begin{equation}
\begin{split}
		f_1(\mathbf Z) = z_1,\quad \varphi_I(f(\mathbf Z)) = Q(\varphi_I(\mathbf Z))\,,
\end{split}
\end{equation}
where $Q$ is a mapping to be defined.
This already yields \eqref{eq:f-symmetric} because by Lemma \ref{phiI-inv},
\begin{equation*}
\begin{split}
		f_1(g \circ\mathbf Z) = g \circ z_1, \quad \text{and} \quad\varphi_I(f(g \circ\mathbf Z)) = Q(\varphi_I(g \circ \mathbf Z)) = Q(\varphi_I( \mathbf Z)) = \varphi_I(g \circ f(\mathbf Z))\,.
\end{split}
\end{equation*}
We now define $Q$. First, if $\mathbf Z$ is not admissible (which depends only on $\varphi_I(\mathbf Z)$), then we let
\begin{equation}
	Q(\varphi_I(\mathbf Z)) := \varphi_I(\mathbf Z)\,.
\end{equation}
Next, if $\mathbf Z$ is admissible, since $\mathbf X$ is also admissible, we know that for all $a \in \{-1,1\}^{k-1}$,
\begin{equation*}
	\big||I_a(\mathbf Z)| - |I_a(\mathbf X)|\big| \leq 2C_2 (N\log N)^{1/2}\,.
\end{equation*}
Hence there exists $(I'_a(\mathbf Z))_{a \in \{-1,1\}^{k-1}}$ such that for all $a \in \{-1,1\}^{k-1}$,
\begin{align}
 \label{eq:1258-2}|I'_a(\mathbf Z)| &= |I_a(\mathbf X)|\,,\\
 \label{eq:1258-1}\text{and}\quad \quad \quad \quad |I_a(\mathbf Z) \triangle I'_a(\mathbf Z)| &\leq 2C_2 (N\log N)^{1/2}\,.
\end{align}
We define
\begin{equation}
	Q(\varphi_I(\mathbf Z)) := (I'_a(\mathbf Z))_{a \in \{-1,1\}^{k-1}}\,.
\end{equation}
This completes the construction of $f$. Next, we verify the remaining properties that $f$ must satisfy.

To verify \eqref{eq:f-close}, first note that if $\mathbf Z$ is not admissible, $f(\mathbf Z) = \mathbf Z$. So \eqref{eq:f-close} holds. If $\mathbf Z$ is  admissible, then for any $j \in \cup_a (I_a(\mathbf Z) \cap I'_a(\mathbf Z))$, the $j$-th component of $z_1 \circ z_i$ is the same as that of $f_1(\mathbf Z) \circ f_i(\mathbf Z)$ for $i = 2,...,k$. Then by $f_1(\mathbf Z) = z_1$ and \eqref{eq:1258-1}, for each $i$, there are at most $2^{k-1} \cdot 2C_2 (N\log N)^{1/2} $ many components where $z_i$ and $f_i(\mathbf Z)$ are not the same. Hence \eqref{eq:f-close} follows. Note that the existences of $g $ and $\sigma$ are ensured by Lemma \ref{number match} and \eqref{eq:1258-2} when $\mathbf Z$ is admissible. We complete the proof of Lemma \ref{adjust}.
\end{proof}
\begin{lemma}Recall $\mathrm{dist}(\cdot, \cdot)$ as in \eqref{eq:def-HMdist}. There exists an absolute constant $C>0$ such that
\label{angle}
$$|H(x) \setminus H(y)| \leq C(\mathrm{dist}(x,y)\log N/N)^{1/2} 2^N\,.$$
\end{lemma}
\begin{proof}
Let $m = \mathrm{dist}(x,y)$ and $Z = (Z_1,...,Z_N)$ be a uniform random vector in $\Sigma_N$ defined on a probability space $(\Omega,\mu)$. Then for any $b \in \R$,

\begin{equation}
	|\{ z \in \Sigma_N : z \cdot x \geq b,z \cdot y < b\}| = 2^N \mu(Z \cdot x \geq b, Z \cdot y < b)\,.
\end{equation}
By symmetry, this equals to
\begin{equation*}
	2^N \mu\Big(\sum_{i=1}^{m} Z_i + \sum_{i=m+1}^NZ_i \geq b, -\sum_{i=1}^{m} Z_i + \sum_{i=m+1}^NZ_i < b \Big) \leq	 2^N \mu\Big(\sum_{i=1}^{m} Z_i \geq \Big| \sum_{i=m+1}^NZ_i -b\Big| \Big)\,.
\end{equation*}
Note that by Hoeffding's inequality,
\begin{equation*}
	\mu\Big(\sum_{i=1}^{m} Z_i \geq (m\log N)^{1/2}\Big) \leq N^{-1/2}\,,
\end{equation*}
and Stirling's formula yields that for some absolute constant $C>0$,
\begin{equation*}
	\mu\Big(\big| \sum_{i=m+1}^NZ_i -b\big| \leq (m\log N)^{1/2} \Big) \leq \mu\Big(\big| \sum_{i=m+1}^NZ_i \big| \leq (m\log N)^{1/2} \Big) \leq C\frac{(m\log N)^{1/2}}{\sqrt{N - m}}\,.
\end{equation*}
The desired result follows from combining previous three inequalities.
\end{proof}

\begin{proof}[\bf Proof of Lemma \ref{gentlemap}]
Let $\mathbf Y^{(i)} = (Y^{(i)}_1,...,Y^{(i)}_k)$, $i = 1,\dots,N_*$ be i.i.d. uniform random vectors in $\Sigma_N^k$ defined on a probability space $(\Omega,\mu)$. 

The idea of the proof is the following. If $q(A)$ is non-negligible, then there will be a certain proportion of $\mathbf Y^{(i)}$'s, for which the intersection $A \cap \bigcap_{j = 1}^kH(f_j(\mathbf Y^{(i)}))$ is empty. This means that for those $\mathbf Y^{(i)}$'s, the vectors $\mathbf Y^{(i)}$'s themselves can remove all the points in $A$ except the difference between $\bigcap_jH(f_j(\mathbf Y^{(i)}))$ and $\bigcap_jH(Y^{(i)}_j)$. And if we use all $\mathbf Y^{(i)}$'s, what is left is at most the intersection of all those differences. Since $f$ is gentle, $f_j(\mathbf Y)$ and $Y_j$ are close and hence each difference must be very small. We will prove that the intersection of the differences is empty with high probability.

We now implement this idea. Let $E$ be the event that the number of $\mathbf Y^{(i)}$'s such that $\bigcap_{j = 1}^kH(f_j(\mathbf Y^{(i)}))$ and $A$ are disjoint is at least $\lfloor q(A) N_*/2 \rfloor$, namely,
\begin{equation}
	E := \Big\{\sum_{i = 1}^{N_*}\11_{\{A \cap \bigcap_{j = 1}^kH(f_j(\mathbf Y^{(i)}))  = \varnothing\}} \geq \lfloor q(A) N_*/2 \rfloor \Big\}\,.
\end{equation}
Note that the events $\{A \cap \bigcap_{j = 1}^kH(f_j(\mathbf Y^{(i)}))  = \varnothing\}$ for $i = 1,\dots,N_*$ are independent and each event has probability $q(A)$ as defined in \eqref{eq:def-qa}. It follows from Chernoff's inequality (see e.g. \cite[Exercise 2.3.5]{vershynin2018high}) that for some absolute constant $c>0$,
\begin{equation}
\label{eq:0409}
	\mu ( E^c) \leq \exp(- c q(A) N_*)\,.
\end{equation}
We now suppose $E$ occurs, and denote those $\mathbf Y^{(i)}$'s by $\mathbf Y^{(i_s)}$ for $s =1,...,m$ with $ m\geq \lfloor q(A) N_*/2 \rfloor$. Let
\begin{equation}
	A_0 := A \cap \bigcap_{i = 1}^{N_*} \bigcap_{j = 1}^k H(Y_j^{(i)}) \,.
\end{equation}
Note that for $s =1,...,\lfloor q(A) N_*/2 \rfloor$,
\begin{equation}
	 A \cap \bigcap_{j = 1}^k H(f_j(\mathbf Y^{(i_s)})) = \varnothing \,.
\end{equation}
Thus, for every $x \in A_0$, we have that $x \in H(Y_j^{(i_s)})$ for all $j \in \{1,...,k\}$, but $ x \not \in H(f_{j'}(\mathbf Y^{(i_s)}))$ for some $j' \in \{1,...,k\}$. Hence
\begin{equation*}
	A_0  \subset \bigcap_{s =1}^{\lfloor q(A) N_*/2 \rfloor} \Big[ \bigcup_{j = 1}^k H(Y_j^{(i_s)}) \cap \bigcup_{j = 1}^k H(f_j(\mathbf Y^{(i_s)})^c \Big] \subset \bigcap_{s =1}^{\lfloor q(A) N_*/2 \rfloor} \Big[ \bigcup_{j = 1}^k \Big(H(Y_j^{(i_s)}) \setminus H(f_j(\mathbf Y^{(i_s)}))\Big) \Big]\,.
\end{equation*}

Next, we prove that the right-hand side of the previous equation is empty with high probability by a first moment computation. Let $\mathbf Y = \mathbf Y^{(1)}$. We claim that for every $z \in \Sigma_N$,
\begin{equation}
\label{eq:0342}
	\mu\Big[z \in \bigcup_{j = 1}^k \Big(H(Y_j)  \setminus H(f_j(\mathbf Y))\Big) \Big] \leq N^{-1/25}.
\end{equation}
To prove \eqref{eq:0342}, we first prove that this probability does not depend on $z$.
For any fixed $g \in \Sigma_N$, it follows from Lemma \ref{autoH} and \eqref{eq:f-symmetric} that this probability equals to
\begin{equation*}
\begin{split}
\mu\Big[g \circ z \in \bigcup_{j = 1}^k \Big(H(g \circ Y_j) \setminus H(g \circ f_j(\mathbf Y))\Big)  \Big]=\mu\Big[g \circ z \in \bigcup_{j = 1}^k \Big( H(g \circ Y_j) \setminus H( f_j(g \circ\mathbf Y))  \Big) \Big]\,.  	
\end{split}
\end{equation*}
Note that $g \circ\mathbf Y$ has the same distribution as $\mathbf Y$. Therefore for an arbitrary $g \in \Sigma_N$,
\begin{equation*}
	\mu\Big[z \in \bigcup_{j = 1}^k \Big(H(Y_j) \setminus H(f_j(\mathbf Y)) \Big) \Big]  = \mu\Big[g \circ z \in \bigcup_{j = 1}^k \Big(H(Y_j) \setminus H(f_j(\mathbf Y)) \Big) \Big]\,.
\end{equation*}
Then it follows that 
\begin{equation}
\label{eq:55556}
	\mu\Big[z \in \bigcup_{j = 1}^k \Big(H(Y_j) \setminus H(f_j(\mathbf Y)) \Big) \Big] = 2^{-N} \E_\mu \Big[ |\bigcup_{j = 1}^k \Big(H(Y_j) \setminus H(f_j(\mathbf Y)) \Big)| \Big]\,. 
\end{equation}
Note that combining \eqref{eq:f-close} and the assumption $k \leq \log N/2$ gives
$$ \mathrm{dist}(Y_j,f_j(\mathbf Y)) \leq N^{9/10}\,.$$
Then by Lemma \ref{angle}, there exists an absolute constant $C>0$ such that
\begin{equation}
	|H(Y_j) \setminus H(f_j(\mathbf Y))| \leq C2^N N^{-1/20} (\log N)^{1/2}\,.
\end{equation}
Combined with \eqref{eq:55556}, it yields \eqref{eq:0342}.
Then, it follows from Chernoff's inequality (for large deviations, see e.g. \cite[Theorem 2.3.1]{vershynin2018high}) and the condition $q(A) \geq (\log N)^{-1/3}$ that for any $z$,
\begin{equation}
	\mu\Big ( \sum_{i = 1}^{N_*}\11_{\Big\{z \in \bigcup_{j = 1}^k \Big(H(Y^{(i)}_j) \setminus H(f_j(\mathbf Y^{(i)}))  \Big) \Big\}} \geq \lfloor q(A) N_*/2 \rfloor\Big) \leq \exp(-c q(A) \cdot \log N \cdot N_*)\,.
\end{equation}
Recall that $N_* = \lfloor N/ (\log N)^{1/2} \rfloor$. We have that
\begin{equation*}
\begin{split}
		\E_\mu \Big[ |A_0| \11_E\Big] &\leq \E_\mu \Big[ \Big| \bigcap_{s =1}^{\lfloor q(A) N_*/2 \rfloor} \Big[\bigcup_{j = 1}^k \Big(H(Y_j^{(i_s)}) \setminus H(f_j(\mathbf Y^{(i_s)})) \Big) \Big]\Big| \11_E\Big] \\
		& \leq \E_\mu \Big| \Big\{ z \in \Sigma_N:\sum_{i = 1}^{N_*}\11_{\Big\{z \in \bigcup_{j = 1}^k \Big(H(Y^{(i)}_j) \setminus H(f_j(\mathbf Y^{(i)}))  \Big) \Big\}} \geq \lfloor q(A) N_*/2 \rfloor \Big\}\Big|\Big]\\
	&\leq 2^N \exp(-c q(A) \cdot \log N \cdot N_*) \leq \exp(-N (\log N)^{1/10})\,.
\end{split}
\end{equation*}
Combined with \eqref{eq:0409}, this gives $$\mu(A_0 \not = \varnothing) \leq \mu(E^c) + \mu(A_0\not = \varnothing, E)\leq \exp(- c q(A) N_*) +\exp(-N (\log N)^{1/10})\,.$$
Hence \eqref{eq:64732} follows.
\end{proof}

\section{Proof of Lemma \ref{ublb-kappa}}
\label{sec:ublb}
For the same reason as in \eqref{eq:com123}, it suffices to prove the following statement. 

Let $(Y_i)_{i \geq 1}$ be a sequence of i.i.d. uniform random vectors in $\Sigma_N$. Then there exist constants $C_3 \geq c_4>0$ and $c>0$ depending only on $\kappa$ such that for sufficiently large $N$,
\begin{align}
	&  \P\Big(\bigcap_{i = 1}^{\lfloor C_3 N \rfloor} H(Y_i) = \varnothing \Big) \leq e^{-c N} \label{eq:ub-tal}\,,\\
	&\P\Big(\bigcap_{i = 1}^{\lfloor c_4 N \rfloor} H(Y_i) = \varnothing \Big) \geq 1 - e^{-c N}\label{eq:lb-tal}\,.
\end{align}

The first inequality \eqref{eq:ub-tal} follows from a first moment computation: Since for all $z \in \Sigma_N$, $\P(z \in H(Y_1)) \leq 2^{-c(\kappa)}$ for some constant $c(\kappa)>0$ only depending on $\kappa$,
\begin{equation}
	\E\big[ \big|\bigcap_{i = 1}^{\lfloor C_3 N \rfloor} H(Y_i) \big|\big] =  2^N \P(z \in H(Y_1))^{\lfloor C_3 N \rfloor} \leq 2^{(1 - c(\kappa) C_3/2) N}\,.
\end{equation}

For the second inequality \eqref{eq:lb-tal}, we can adapt the proof of \cite[Theorem 1.3]{Tala99}, which gives \eqref{eq:lb-tal} for $\kappa = 0$. In fact, the proof of \cite[Theorem 1.3]{Tala99} only used $\kappa = 0$ in \cite[Lemma 4.2]{Tala99}, which corresponds to the following lemma when $\kappa = 0$.
	\begin{lemma}
	\label{tal-lemma}
	Let $\xi$ be a uniform random vector in $\Sigma_N$. There exist constants $K,\gamma >0$ depending only on $\kappa$ with the following property: given $\sigma^1,...,\sigma^r$ in $\Sigma_N$ such that 
	\begin{equation}
			\text{for all } p,q \leq r, \quad p \not = q \implies \frac{1}{N}\sum_{i \leq N}\sigma_i^p\sigma_i^q \leq 1/2\,,
		\end{equation}	
	we have
	\begin{equation}
	\label{eq:tal-lemma}
		\P(\sigma^p \not \in H(\xi), p = 1,\dots,r) \leq K r^{-\gamma}\,.
	\end{equation}
	\end{lemma}
Therefore, to prove \eqref{eq:lb-tal}, it suffices to prove Lemma \ref{tal-lemma} for all $\kappa \in \R$. And the rest of the proof \cite[Theorem 4.1, Corollary 4.3, Proof of Theorem 1.3]{Tala99} can be carried out exactly as in \cite{Tala99}. In fact, it is straight forward to extend \cite[Lemma 4.2]{Tala99} by modifying its original proof as follows. 

	Consider the random variable \begin{equation}
			Y = \sup_{p \leq r} \xi \cdot \sigma^p = \sup_{p \leq r} \sum_{i \leq N}\xi_i \sigma_i^p\,.
		\end{equation}
As in \cite[Lemma 4.2]{Tala99}, using concentration of measure for Bernoulli random variable \cite[(2.17)]{Tal93} yields that for some absolute constant $K_1 >0$,
		\begin{equation}
		\label{eq:Tal-1}
			\P(Y \leq \E Y - K_1(1 + u) \sqrt{N}) \leq 4\exp(-u^2)\,.
		\end{equation}
In addition, it is proved in \cite[Lemma 4.2]{Tala99} by a version of Sudakov minoration for Bernoulli random variable that for some absolute constant $K_2 > 0$,
		\begin{equation}
		\label{eq:Tal-2}
			\E Y \geq \frac{\sqrt{N}}{K_2} \sqrt{\log r}\,.
		\end{equation} 
		We now let
	\begin{equation}
		u := \E Y/(4K_1\sqrt{N})\,.
	\end{equation}
Then by \eqref{eq:Tal-1} and \eqref{eq:Tal-2}, we have
			\begin{equation*}
			\P(Y \leq \E Y - K_1(1 + u) \sqrt{N}) \leq 4\exp(-(\E Y)^2/(16K_1^2N)) \leq 4\exp(-\log r/(16 K_1^2K_2^2))\,.
		\end{equation*}
For sufficiently large $r$ depending only on $(\kappa,K_2,K_1)$,
	\begin{equation*}
		K_1(1 + u) \sqrt{N} \leq 2 K_1 u \sqrt{N} = \E Y/2 < \E Y - \kappa \sqrt{N}\,.
	\end{equation*}
Combining previous two inequalities yields \eqref{eq:tal-lemma}, and thus we complete the proof of Lemma \ref{ublb-kappa}.

\section{Proof of Proposition \ref{01Sharpthm}}
\label{sec:STT}

This proposition is an improvement of \cite[Corollary 2.10]{H12} and a large part of the proof below originates from the proof of \cite[Corollary 2.10]{H12}.

Recall $J_\mathcal J$, $\mathcal{F}_\mathcal{J}$ and pseudo-junta as defined around Definition \ref{pseudo-junta}. It was shown in \cite[Theorem 2.8]{H12} that there exist a collection of boolean functions $\mathcal{J}=\{J_S\}_{S \subseteq [n]}$ and
a pseudo-junta $h:\{0,1\}^n \rightarrow \{0,1\}$ measurable with respect to $\mathcal{F}_\mathcal{J}$, such that  
\begin{equation}
\label{eq:J and h}
	\E_p[ |J_\mathcal{J}(x)|] \le e^{10^{11} \lceil I_f \rceil^2 (\delta\epsilon_N)^{-2}}, \quad \text{and} \quad \mu_p( f \not =  h) \le \delta\epsilon_N/2\,.
\end{equation}
Set $k:= e^{10^{12} \lceil I_f \rceil^2  (\delta\epsilon_N)^{-2}}$, and let
\begin{equation}
\label{eq:321max}
	\beta := \max_{(S,y)} \E_p \left[f(x) \mid x_{S}=y, J_\mathcal{J}(x)=S,  \sfb_{\delta,k}^c \right],
\end{equation}
where the maximum is over $(S, y)$ such that
\begin{equation}
	\label{eq:SyB}S \subset [n], \quad |S| \leq k, \quad y \in \{0,1\}^S, \quad \{x_{S}=y, J_\mathcal{J}(x)=S\} \cap  \sfb_{\delta,k}^c \not = \varnothing \,. 
\end{equation}

We first prove that $\beta < 1- \delta$. Suppose $\beta \geq 1- \delta$, then there exists a subset $S \subseteq [n]$ with $|S| \le k$ and $y_0 \in \{0,1\}^S$ such that
\begin{equation}
\label{eq:denseAtom}
\E_p \Big[f(x) \mid x_{S}=y_0, J_\mathcal{J}(x)=S,\sfb_{\delta,k}^c \Big]\geq 1- \delta.
\end{equation}
As in \cite[Corollary 2.10]{H12}, we define $h_1, h_2:\{0,1\}^{[n]\setminus S} \rightarrow \{0,1\}$ as $h_1:x \mapsto f(y_0,x)$, and $h_2:x \mapsto 1_{[J_\mathcal{J}(y_0,x)=S]}$. Then we can rewrite (\ref{eq:denseAtom}) as
\begin{equation}
\label{eq:denseAtomRewrite}
\E_p \Big[h_1(x_{[n] \setminus S}) \mid h_2(x_{[n] \setminus S})=1 ,\sfb_{\delta,k}^c\Big]\geq 1- \delta\,.
\end{equation}
Note that that $f$ is an increasing function implies that $h_1$ is increasing. And it was proved in the proof of \cite[Corollary 2.10]{H12} that we can assume that  $h_2$ is decreasing (because we can assume $J_S$ are increasing by \cite[Remark 2.9]{H12} for $p\leq 1/2$.) Also, it is clear that $\11_{\sfb_{\delta,k}^c}$ is decreasing. Then it follows from FKG inequality that
$$ \E_p \Big[f(x) \mid x_{S}=y_0 \Big] = \E_p[h_1(x_{[n] \setminus S})] \geq \E_p \Big[h_1(x_{[n] \setminus S}) \mid h_2(x_{[n] \setminus S})=1,\sfb_{\delta,k}^c \Big] \geq 1- \delta.$$
Let $S' \subset S$ be the coordinates of $y_0$ with value $1$. Since $f$ is increasing,
\begin{equation*}
	\E_p \Big[f(x) \mid x_{S'}=(1,\ldots,1) \Big]  \geq \E_p \Big[f(x) \mid x_{S'}=(1,\ldots,1),x_{S \setminus S'} = (0,\ldots,0)  \Big]\geq 1- \delta\,.
\end{equation*}
This means that $\{x_{S} = y_0\} \subset \sfb_{\delta,k}$, which contradicts with \eqref{eq:SyB}.

Next, note that
\begin{equation*}
\begin{split}
	  	\mu_p(f = 1, \sfb_{\delta,k}^c) \leq \mu_p(|J_\mathcal{J}| \leq k,f = 1,h=0,\sfb_{\delta,k}^c)+ \mu_p(|J_\mathcal{J}| \leq k,h=1,\sfb_{\delta,k}^c) + \mu_p(|J_\mathcal{J}| > k)\,.
\end{split}
  \end{equation*}
Since $f$ and $h$ are Boolean and $h$ is measurable with respect to $\mathcal{F}_\mathcal{J}$, \eqref{eq:321max} implies
\begin{equation*}
	 \mu_p(f = 0 , |J_\mathcal{J}| \le k, h = 1,\sfb_{\delta,k}^c) \geq (1 - \beta) \mu_p(|J_\mathcal{J}| \le k, h = 1,\sfb_{\delta,k}^c)\,.
\end{equation*}
Combining previous two inequalities gives
\begin{equation*}
	\mu_p(f = 1, \sfb_{\delta,k}^c) \leq (1 - \beta)^{-1} \mu_p(f \not = h, \sfb_{\delta,k}^c) + \mu_p(|J_\mathcal{J}| > k)\,.
\end{equation*}
Combined with \eqref{eq:J and h} and $\beta < 1 - \delta$, this completes the proof of Proposition \ref{01Sharpthm}.\qed
\section*{Acknowledgments}
I would like to thank Jian Ding for suggesting the problem and carefully reviewing the draft of the paper.
\small
\bibliography{sharp}
\bibliographystyle{abbrv}
\end{document}